\documentclass[12pt]{amsart}
\usepackage{amsmath}	
\usepackage{amssymb,latexsym, amsmath}
\usepackage{enumerate}
\usepackage{hyperref}
\usepackage{verbatim}
\usepackage{fullpage}

\makeatletter \@namedef{subjclassname@2010}{%
  \textup{2010} Mathematics Subject Classification}
\makeatother

\newcounter{thm} \numberwithin{thm}{section}
\newtheorem{Theorem}[thm]{Theorem}

\newtheorem{Lemma}[thm]{Lemma}
\newtheorem{Corollary}[thm]{Corollary}

\newcommand{\NN}[0]{\mathbb N}
	
\newcommand{\CC}[0]{\mathbb C}	
	\newcommand{\QQ}[0]{\mathbb Q}
\newcommand{\EE}[0]{\mathbb E}	\newcommand{\RR}[0]{\mathbb R}
\newcommand{\FF}[0]{\mathbb F}	
	\newcommand{\TT}[0]{\mathbb T}
	\newcommand{\vv}[0]{\mathbb U}

	\newcommand{\ZZ}[0]{\mathbb Z}

	\newcommand{\cP}[0]{\mathcal P}

\newcommand{\pp}[0]{\textbf{\textit{p}}}

\renewcommand{\vv}[0]{\textbf{\textit{v}}}

\renewcommand{\bar}[1]{\overline{#1}}
\newcommand{\one}[0]{\mathbf{1}}
\newcommand{\eps}[0]{\varepsilon}

\renewcommand{\mod}[1]{\ (\text{mod }#1)}

\newcommand{\lr}[1]{\left(#1\right)}

\renewcommand{\gcd}[0]{\mathrm{gcd}}

\renewcommand{\hat}[1]{\widehat{#1}}

\newcommand{\ang}[1]{\left\langle#1\right\rangle}

\begin{document}

\baselineskip=17pt

% %%%%%%%%%%

% % In the running head, replace first names by initials % and give an
% abbreviation of the title.

\title{The sum-product problem for integers with few prime factors}  

\author[B. Hanson]{Brandon Hanson} \address{University of Maine, Orono, Maine, U.S.A.}
\email{brandon.w.hanson@gmail.com}

\author[M. Rudnev]{Misha Rudnev} \address{School of Mathematics, University of Bristol, Bristol BS8 1UG, U.K.}
\email{misharudnev@gmail.com}

\author[I. Shkredov]{Ilya Shkredov} 
%%\address{London Institute for  Mathematical Sciences, 21 Albemarle St., London W1S 4BS, U.K.}
\email{ilya.shkredov@gmail.com}

\author[D. Zhelezov]{Dmitrii Zhelezov} 
\address{Johann Radon Institute for Computational and Applied Mathematics\\
Linz, Austria}
\email{dzhelezov@gmail.com}

\date{}
\maketitle

\begin{abstract}
    It was asked by E. Szemer\'edi if, for a finite set $A\subset\ZZ$, one can improve estimates for $\max\{|A+A|,|A\cdot A|\}$, under the constraint that all integers involved have a bounded number of prime factors -- that is, each $a\in A$ satisfies $\omega(a)\leq k$. In this paper, answer Szemer\'edi's question in the affirmative by showing that this maximum is of order $|A|^{\frac{5}{3}-o(1)}$ provided $k\leq (\log|A|)^{1-\eps}$ for some $\eps>0$. In fact, this will follow from an estimate for additive energy which is best possible up to factors of size $|A|^{o(1)}$.
\end{abstract}

\maketitle

\section{Introduction}

    The sum-product phenomena was introduced by Erd\H{o}s and Szemer\'edi in \cite{ES}.
   \begin{quotation}
 {\em  Let  $1 < a_1 < \ldots  <a_n $  be a sequence of integers. Consider the integers of the form \begin{equation}\label{eq:quote} a_i+a_j,\;a_ia_j: \qquad 1\leq i\leq j\leq n\,. \end{equation}
It is tempting to conjecture that for every $\epsilon >0$ there is an $n_0$, so that for every $n\geq n_0$, 
there are more than  $n^{2-\varepsilon}$  distinct integers of the form \eqref{eq:quote}.} \end{quotation}

In contemporary notation, we are interested in the sizes of the set of sums and products, defined for a subset $A$ of integers, or more generally a commutative ring, as
\[A+A=\{a+b:a,b\in A\},\ A\cdot A=\{ab:a,b\in A\}.\]
As a general heuristic, the conjecture suggests that either $A+A$ or $A\cdot A$ is significantly larger then the original set, unless $A$ is close to a subring. In the case of the integers, the latter cannot occur as there are no non-trivial finite subrings. The interested reader may consult \cite{TaoVu} for a rather thorough treatment of sumsets and related questions, including some prior work on the sum-product problem.

Erd\H{o}s and Szemer\'edi  continue with the following statement. \begin{quotation}  {\em Perhaps our conjectures remain true if the   $a'$s  are real or complex numbers.} \end{quotation} Erd\H{o}s and Szemer\'edi ultimately prove
\[
    \max\{ |A+A|, |A\cdot A| \} \gg |A|^{1+c} \,,
\]
 where the exponent $c$ can be seen to equal to $\frac{1}{31}$, see \cite{N}, and conjecturally, any $c<1$ is admissible, at the cost of the implicit constant. 
 
 The sum-product phenomenon has been extensively studied in the last few decades, the current records as of writing being
\cite{RS} for real numbers, and \cite{MSt} (also, see \cite{RSh}) for sufficiently small sets in finite fields. 

While the sum-product problem was originally posed for finite sets of integers, a number of techniques involving combinatorial and convex geometry have been the predominant tools in the area a number of years, and these techniques work just as well for finite sets of reals. Indeed, see \cite{E}, \cite{S}, and \cite{HRNR} for techniques that helped establish our current understanding of the problem over the reals. However, one aspect of the problem that is understood in only the arithmetic setting (over $\ZZ$ or perhaps $\QQ$) is the nature of sets with few distinct products. Indeed, in this setting, unique factorization and $p$-adic analysis have allowed for progress which has not been matched by real-variable methods. Results leveraging the techniques we have in mind begin with \cite{Chang} and the subsequent \cite{BC}, and have been elaborated upon in \cite{HRNZ1}, \cite{HRNZ2} and \cite{PZ}. In particular, one has a much better understanding of sets $A$ for which $A\cdot A$ is very small when $A$ consists of integers. Perhaps motivated by these sorts of results, and the fact that $\{1,\ldots,N\}$ is a near extremal example for the sum-product problem, E. Szemer\'edi asked the fourth listed author whether sum-product estimates for $A$ are improved when $A$ consists of integers satisfying\footnote{Here, $\omega(a)$ denotes the number of distinct prime factors of $a$.} $\omega(a)\leq k$, and even with the very limiting constraint, say, $k=10$. This is a natural question to ask, as if something like the initial segment $A=\{1,\ldots,N\}$ were in fact the worst case, then one would have $\omega(a)\leq \log\log|A|(1+o(1))$, on average -- see, for instance, \cite{MV}.

Even when $k=1$, this problem already hints at some subtle behaviour. Indeed when $A$ consists of the primes up to $N$, of which there are about $N/\log N$ by the \textit{Prime Number Theorem}. One of course has  $|A\cdot A|\gg |A|^2$, but $|A+A|\ll |A|\log|A|$; more dramatic still is when $A$ is an arithmetic progression composed of primes, whose existence is the content of the \textit{Green-Tao Theorem} proved in \cite{GT}, where one has $|A+A|= 2|A|-1$. So the constraint $k=1$ still allows for $A$ to be as structures as is possible so far as addition is concerned. On the other hand, one could also take $A=\{p,p^2,\ldots,p^N\}$ for any prime $p$ -- recall that $\omega$ ignores multiplicity -- and so $|A+A|\gg |A|^2$ but $|A\cdot A|= 2|A|-1$, and now the multiplicative structure of $A$ is as maximal.

It is really in the case $k=2$, however, that the problem shows its teeth. The following example, which we know from \cite{BW} is what we shall refer to as a \textit{Balog-Wooley set}. It merely consists of of the product of a geometric progression and an arithmetic progression: \begin{align*}
\Gamma&=\{r^m:1\leq m\leq M\},\\B&=\{a+dn:1\leq n\leq N\},\\A&=\Gamma\cdot B=\{r^m(a+dn):1\leq m\leq M,\ 1\leq n \leq N\}.
\end{align*}
Balog and Wooley chose $B=\{1,\ldots,2n^2\}$ and $\Gamma=\{1,2,4,\ldots,2^{n-1}\}$. (Note that to avoid collisions one can replace the prime $2$ generating $\Gamma$ by some $p>2n^2$.) More generally, one can consider \textit{approximate Balog-Wooley sets} where $\Gamma$ and $B$ are replaced by approximations to geometric and arithmetic progressions, and one can certainly choose $r$, as well as all members of $B$ to consist be primes, as described in the $k=1$ case. These examples gain their structure from imposing the constraints \[|\Gamma\cdot \Gamma|\leq K_\Gamma|\Gamma|,\ |B+B|\leq K_B|B|\] with $K_\Gamma$ and $K_B$ basically as small as desired. This will force the product set to be small, since
\[|A\cdot A|\leq |\Gamma\cdot \Gamma||B\cdot B|\leq K_\Gamma|\Gamma||B|^2\leq (K_\Gamma|B|)|A|.\]
At the same time, these sets also satisfy
\[E_+(A,A)\geq |\Gamma|E_+(B,B)\geq \frac{|\Gamma||B|^3}{K_B}=\frac{|A|^3}{K_B|\Gamma|^2},\]
where $E_+(X,Y)$ denoted the \textit{additive energy}
\[E_+(X,Y)=|\{(x_1,x_2,y_1,y_2)\in X\times X\times Y\times Y:x_1+y_1=x_2+y_2\}|,\] owing merely to the diagonal solutions \[\gamma(b_1+b_2)=\gamma(b_3+b_4) .\]
In particular, $A$ can satisfy $|A\cdot A|\ll |A|^{\frac{5}{3}-o(1)}$ and $E_+(A,A)\gg |A|^{\frac{7}{3}}$ by taking $|\Gamma|^2=|B|$, as did Balog and Wooley in their work. They further conjectured that the exponent $\frac{7}{3}-o(1)$ was best possible should one wish to decompose $A$ into pieces, one with small additive energy and the other with small multiplicative energy. If one used only additive energy to predict $|A+A|$, say by the standard Cauchy-Schwarz estimate \[|A+A|E_+(A,A)\geq |A|^4,\]
then one could only deduce $|A+A|\gg |A|^{\frac{5}{3}}$. From this example, we observe that even if each member of $A$ has but two prime factors, one can do no better than the exponent $\frac{5}{3}$ in the statement \[\textit{There is a subset } \tilde A\subseteq A \textit{ for which } |A\cdot A| + \frac{|\tilde A|^4}{E_+(\tilde A, \tilde A)} \gg |A|^{\frac{5}{3}},\]
losing another $o(1)$ to the exponent $\frac{5}{3}$ if the number of prime factors increases to roughly $\log\log |A|$.

In fact, it is this statement that we prove, up to terms growing slower than any power of $|A|$ and under the few prime factor constraint. Of course, the additive energy of $A$ may not be an accurate predictor for $|A+A|$, and indeed, Balog-Wooley sets do not violate the Erd\H{o}s-Szemer\'edi conjecture, which \textit{is} still a conjecture, after all.

Our main theorem is the following.

\begin{Theorem}\label{MainThm}
    Let $\eps$ be a real number with $0<\eps<1/6$ and let $k\in \NN$ be positive integer. Suppose $A$ is a sufficiently large finite set of integers satisfying $\omega(a)\leq k$ for each $a\in A$ and such that $(\log |A|)^{1-6\eps}\geq k$. Then, there is a subset $\tilde A\subseteq A$ with \[|\tilde A|\geq \frac{|A|}{k! (2\log|A|)^k}\] and \[|A\cdot A|+\frac{|\tilde A|^4}{E_+(\tilde A,\tilde A)}\geq |A|^{\frac{5}{3}}\exp(-C(\log |A|)^{1-\eps}),\]
    for some absolute constant $C>0$.
\end{Theorem}

In particular, we get the following sum-product estimate.

\begin{Corollary}
    Let $A$ be a finite set of integers and let $\eps$ and $k$ be parameters such that the conditions of Theorem \ref{MainThm} are met. Then, 
    \[\max\{|A+A|,|A\cdot A|\}\gg_{\eps,k} |A|^{\frac{5}{3}-o(1)}.\]
\end{Corollary}

As remarked above, the exponent $\frac{5}{3}$ is no coincidence, as our approach begins by showing that Balog-Wooley sets are essentially the worst possible scenario. It is because of this approach, namely an attack on the additive energy of $A$, that we fail to prove full quadratic growth. 

We also note that most of the sum-product estimates that we are aware of juxtapose energy with cardinality of ``the opposite set".
This is explicit in the title of the paper by Solymosi \cite{S}, which proves the inequality
$$
E_\times(A,A) \ll |A+A|^2 \log |A|\,,
$$
with the multiplicative energy $E_\times(A,A)$ being defined analogously to additive energy, with multiplication in place of addition. This, in particular, implies that if $|A+A|\leq K_+ |A|$, with a sufficiently small {\em additive doubling constant} $K_+$, then $E^\times(A,A)$ barely exceeds its trivial lower bound $|A|^2$ and hence $|A\cdot A|$ is almost $|A|^2$, proving a case of the Erd\H os-Szemer\'edi conjecture for reals (and even complex) numbers at the endpoint $K_+\approx 1$, see also \cite{KR}. 

On the other end, with $K_\times=|A\cdot A|/|A|$ being the {\em multiplicative doubling constant},   P\'alv\H{o}lgyi and the fourth author \cite{PZ} prove that for finite sets $A\subset \ZZ$ and $\eps\in(0,1/2)$, there is a subset $\tilde A\subseteq A$ such that $|\tilde A|\geq |A|^{1-\eps}$ and 
%%IS 2/\eps -> 4/\eps
\[E_+(\tilde A, \tilde A)\leq K_\times^{\frac{4}{\eps}}|A|^{2+4\eps}.\] This theorem shed new light 
 on, and improved the results of Bourgain and Chang \cite{BC}. It also settles the Erd\H os-Szemer\'edi conjecture in the endpoint case $K_\times \approx 1$.\footnote{This is only known for integers (and rationals, by dilation invariance); the best known results for reals is much weaker \cite{BRS}.} Here one gets an additive energy estimate via the product set, which works as a good predictor for $|A+A|$. However, away from the endpoints, and in view of the Balog-Wooley example, an immediate application of the Cauchy-Schwarz inequality to pass from an upper bound on additive energy to a lower bound for sumset is too costly.

As far as the results in this article are concerned, their fountainhead is the following elementary lemma, which illustrates the strength of the few-prime-factors hypothesis.

\begin{Lemma}\label{ProductBound}
    Let $A$ and $B$ be finite sets of integers such that $\omega(a)\leq k$ for $a\in A$ and $\omega(b)\leq l$ for $b\in B$. Then any element $q\in A\cdot B$ admits at most $2^{k+l}$ solutions to $q=ab$ with $a\in A$, $b\in B$ and $\gcd(a,b)=1$.
\end{Lemma}
\begin{proof}
    Indeed, $\omega(q)\leq k+l$ and factorization $q=ab$ with $\gcd(a,b)=1$ amounts to choosing some subset of the primes dividing $q$. There are $2^{\omega(q)}$ such subsets.
\end{proof}

The proof of Theorem \ref{MainThm} involves three steps. First we establish an approximate structure theorem, Theorem \ref{RegularFibredStructure} of Section 2, for sets of integers with few prime factors. It is this theorem that tells us that (generalized) Balog-Wooley sets are essentially worst-possible. Lemma \ref{ProductBound} will combine with the structure theorem to provide uniform control on the fibres of the resulting decomposition.

The next ingredient is a Littlewood-Paley type theorem which allows us to estimate the additive energy of Balog-Wooley type sets quite efficiently. Our Theorem \ref{SquareFunction} is a refinement of the crucial Lemma 6 (see also Lemma \ref{EasyChang} below) of Chang from \cite{Chang}. This lemma was also the basis for developments in subsequent papers \cite{BC}, \cite{HRNZ1}, \cite{HRNZ2}, \cite{PZ}. The statement in question was strong enough to meet the objectives of the above mentioned papers, which dealt with the endpoint case of the small product set. A stronger result is needed to move away from the endpoint.

Theorem \ref{SquareFunction} is deduced from Burkholder's Littlewood-Paley Theorem for martingales, \cite{Bur}. The approach is a natural one, as randomization (i.e. martingale transforms) provides a path to iterated Littlewood-Paley decomposition, and Chang's Lemma is a martingale-difference method if one chooses to view it as such. During the course of this project, we discovered, by way of the wonderful book \cite{Pisier}, that Gundy and Varopoulos had previously observed the sort of result needed, \cite{GV}. We suspect the theorem, which appears as Corollary \ref{EnergyDecomposition} in Section 3, could have further applications to the area. 

The final component of the proof is a bound for additive energy averaged over dilates from a low-rank group. We give an elementary proof of the desired estimate when the rank is 1, see Theorem \ref{RankOne} of Section 4, and appeal to heavier machinery for the more general case, namely Lemma 2.1 of \cite{RNZ}, which relies on a rather strong version of the \textit{Subspace Theorem} from transcendence theory. We are curious as to whether or not our simple approach in Theorem \ref{RankOne} can be made to work in the general case.

Hence, our approach combines basic combinatorics and elementary number theory with modern tools from harmonic analysis and transcendence theory, which happen to perfectly fit into our considerations. One naturally wonders as to what extent this may be a coincidence.

Before proceeding with the proof proper, we give a simpler argument in the Section 3 which is still good enough to achieve the exponent $3/2$.

\section*{Acknowledgements}
The authors thank the Heilbronn Institute for Mathematical Research (HIMR) for funding a Focused Research Group {\em Testing Additive Structure} in May-June 2022, where this project was incepted, and the Johann Radon Institut (RICAM) Linz for being the venue and additional funding provided. We personally thank Oleksiy Klurman for co-organising the FRG and Oliver Roche-Newton for hosting it. Brandon Hanson is supported by NSF Award 2135200. Dmitrii Zhelezov was supported by the Austrian Science Fund FWF Project P 34180. 

\section{Notation}

\textit{Asymptotic notation}: For non-negative quantities $X$ and $Y$, $X\ll Y$ or $X=O(Y)$ both mean that for some absolute constant $C>0$, we have $X\leq CY$, while $X=o(Y)$ means that $X/Y\to 0$. Dependence of the implicit constant $C$ on parameters are indicated by subscripts, so fpr instance $X\ll_{\alpha,\beta} Y$ or $X=O_{\alpha,\beta}(Y)$ means that $X\leq CY$ for some $C(\alpha,\beta)>0$. The symbol $X\approx Y$ is meant to reflect that $X\ll Y\ll X$.

\textit{Number theoretic notation}: We reserve the letter $p$ for primes, and the fundamental theorem of arithmetic is then that, for integral $n$
\[n=\prod_p p^{v_p(n)},\]
the product being over distinct primes, where $v_p(n)$ denoted the $p$-adic valuation of $n$ (i.e. the exponent of $p$ in the unique factorization of $n$). This product is implicitly finite as $v_p(n)=0$ for all but finitely many values of $p$. For those $p$ where $v_p(n)>0$, we write $p|n$ and we let $\omega(n)$ denote the number of such $p$. We write $\gcd(m,n)=\prod_{p}p^{\min(v_p(m),v_p(n))}$ for the greatest common divisor of $m$ and $n$. If $P=\{p_1,\ldots,p_r\}$ is a set of primes, we can fix the ordering as an $r$-tuple $\pp=(p_1,\ldots,p_r)$ and then for an $r$-tuple $\vv=(v_1,\ldots,v_r)$ of integers, we write $\pp^\vv=p_1^{v_1}\cdots p_r^{v_r}$. We shall write $\ang{P}=\{\pp^\vv:\vv\in\ZZ^r\}$ for the multiplicative subgroup of $\QQ^\times$ generated by the primes in $P$, and $\ang{P}_+$ for the multiplicative semigroup containing those elements $n\in \ang P$ with $v_p(n)\geq 0$ for each prime in $P$.

\textit{Additive combinatorial notation}: For subsets $A$ and $B$ of complex numbers, we write 
\[A+B=\{a+b:a\in A,\ b\in B\}\] for the sumset,
\[A\cdot B=\{a+b:a\in A,\ b\in B\}\] for the product set, and 
\[t+A=\{t+a:a\in A\},\ d\cdot A=\{da:a\in A\}\]
for the translate of $A$ by $t\in \CC$ and dilate of $A$ by $d\in \CC$ respectively. The quantities
\[E_+(A,B)=|\{(a_1,b_1,a_2,b_2)\in A\times B\times A\times B:a_1+b_1=a_2+b_2\}|\] and 
\[E_\times (A,B)=|\{(a_1,b_1,a_2,b_2)\in A\times B\times A\times B:a_1 b_1=a_2 b_2\}|\]
denote, respectively, the additive and multiplicative energies of $A$ and $B$.

\textit{Graph theoretic notation}: If $X$ and $Y$ are sets, we will refer to $G\subseteq X\times Y$ as a bipartite graph. It is of course a directed graph, but this point will not be emphasized. We further define, for subsets $X'\subseteq X,\;Y'\subseteq Y$, \[N_{Y'}(x)=\{y\in Y':(x,y)\in G\},\ N_{X'}(y)=\{x\in X':(x,y)\in G\},\]
called the neighbours of $x$ and $y$ in $Y'$ and $X'$ respectively. The cardinalities $|N_Y(x)|$ and $|N_X(y)|$ are the degrees of $x$ and $y$. 

\textit{Analytic notation}: We identify $\TT=\RR/\ZZ=[0,1)$ as the torus, and endow the sufficiently nice functions on $\TT$ with the $L^q$-norm \[\|f\|_{L^q}=\lr{\int_0^1 |f(t)|^qdt}^{1/q}\] for $q\geq 1$ and the inner product \[\langle f,g\rangle=\int_0^1 f(t)\bar{g(t)}dt.\] A function $f:\TT\to\CC$ 
 of the form \[f(t)=\sum_{n\in \ZZ}\hat f(n)e(nt)\] is called a Fourier series. Here, the functions $e(nt)=e^{2\pi int}$ are the standard characters on $\ZZ$ and the coefficient $\hat f(n)$ is the $n$'th Fourier coefficient of $f$. If $\hat f(n)$ is non-zero for finitely many $n$, $f$ is called a trigonometric polynomial. 

\section{A simple argument yielding exponent 3/2}

We begin this section with a rather simple structure theorem that will give a feel for the general problem. It will not be this version of the structure theorem that is ultimately used, but it is pretty simple to prove and illustrates the general strategy.

\begin{Lemma}\label{GraphLemma}
    Suppose $G\subseteq X\times Y$. If $|N_Y(x)|\leq k$ for each $x\in X$, then there is a subset $Y'\subseteq Y$ of size at most $2k^2$ and such that for at least $|X|^2/2$ pairs $(x,x')\in X$, we have $N_Y(x)\cap N_Y(x')\subseteq Y'$. 
\end{Lemma}
\begin{proof}
Indeed, let $Y'=\{y\in Y:|N_X(y)|\geq |X|/2k\}$. Then
\[|Y'|\leq \frac{2k}{|X|}\sum_{y\in Y}|N_X(y)|=\frac{2k}{|X|}\sum_{x\in X}|N_Y(x)|\leq 2k^2.\] Let $Y''=Y\setminus Y'$ 
and observe that
\[\sum_{x,x'\in X}|N_{Y''}(x)\cap N_{Y''}(x')|=\sum_{y\in Y''}|N_X(y)|^2\leq \frac{|X|}{2k}\sum_{x\in X}|N_Y(x)|\leq \frac{|X|^2}{2},\]
and so at most half of the pairs $(x,x')\in X\times X$ have a common neighbour outside of $Y'$.
\end{proof}

\begin{Corollary}\label{FibredStructure}
    Let $A$ be a finite set of integers such that $\omega(a)\leq k$ for $a\in A$. Then there is a set $P=\{p_1,\ldots,p_r\}$ of at most
    %%IS: k -> 2k^2
    $r\leq 2k^2$ primes such that 
    \[A=\bigcup_{\vv\in \NN_0^r}\pp^\vv\cdot B_\vv\]
    where $B_\vv$ is a set of integers prime to $p_1p_2\cdots p_r$ and such that \[\sum_{\vv,\vv'}|\{(b,b')\in B_\vv\times B_{\vv'}:\gcd(b,b')=1\}|\geq \frac{|A|^2}{2}.\]
\end{Corollary}
\begin{proof}
    Let $X=A$ and $Y=\{p:p|a\text{ for some }a\in A\}$, and let $G=\{(a,p)\in X\times Y:p|a\}$. Applying Lemma \ref{GraphLemma}, we find $P$ with $|P|=r\leq 2k^2$ such that for at least half of all pairs $(a,a')\in A\times A$, all primes dividing both $a$ and $a'$ belong to $P$. Each $a\in A$ factors as
    \[a=b_a\prod_{p\in P}p^{v_p(a)}\] with $b_a$ coprime to each $p\in P$, and we thus partition $A$ according to the valuations $v_{p}(a)$ with $p\in P$ and we get \[A=\bigcup_{\vv\in \NN_0^r}\pp^\vv\cdot B_\vv\] for some finite sets of integers $B_\vv$ which are coprime to each $p\in P$. Furthermore, if $(a,a')$ is a pair for which all common prime factors do belong to $P$, then $b_a$ and $b_{a'}$ are coprime.
\end{proof}

In this way, we have taken an arbitrary set $A$ of integers and extracted from it a pseudo-product structure -- one factor of which is from a low-rank multiplicative group, the other of which is multiplicatively independent in the sense that the fibres are relatively prime. 
 Let $\vv_1$ be fixed, and observe that the sets 
\[\pp^{\vv_1}B_{\vv_1}\cdot \pp^{\vv_2}B_{\vv_2}=\pp^{\vv_1+\vv_2}B_{\vv_1}\cdot B_{\vv_2}\]
are disjoint as $\vv_2$ varies, since they are graded by the exponents $\vv_1+\vv_2$. Now suppose that there are $M(\vv_1,\vv_2)$ pairs $(b_1,b_2)\in B_{\vv_1}\times B_{\vv_2}$ which are relatively prime. Then,
\[\sum_{\vv_1}\sum_{\substack{\vv_2\\ |B_{\vv_1}|\geq|B_{\vv_2}|}}M(\vv_1,\vv_2)\geq \frac{1}{2}\sum_{\vv_1,\vv_2}M(\vv_1,\vv_2)\geq \frac{|A|^2}{4},\]
from the conclusion of Corollary \ref{FibredStructure}. But
\[\sum_{\vv_1}|B_{\vv_1}|\sum_{\substack{\vv_2\\ |B_{\vv_2}|\leq |B_{v_1}|}}\frac{M(\vv_1,\vv_2)}{|B_{\vv_1}|}\leq |A|\max_{\vv_1}\frac{1}{|B_{\vv_1}|}\sum_{\substack{\vv_2\\ |B_{\vv_2}|\leq |B_{v_1}|}}M(\vv_1,\vv_2),\]
so that for some choice of $\vv_1$, we have \[\sum_{\substack{\vv_2\\ |B_{\vv_2}|\leq |B_{v_1}|}}M(\vv_1,\vv_2)\geq \frac{|A||B_{\vv_1}|}{4}.\]
Let $V'=\{\vv_2:|B_{\vv_2}|\leq |B_{\vv_1}|\}$ and \[A'=\bigcup_{\vv_2\in V'}\pp^{\vv_2}\cdot B_{\vv_2},\]
so that \[
    |A'|\geq \sum_{\vv_2\in V'}|B_{\vv_2}|\geq \sum_{\vv_2\in V'}\frac{M(\vv_1,\vv_2)}{|B_{\vv_1}|}\geq \frac{|A|}{4}.
\]
Further observe that on coprime pairs, the map $(b_1,b_2)\mapsto b_1b_2$ is at most $4^{k}$-to-one by Lemma \ref{ProductBound}, so that \begin{equation}\label{ThreeHalvesProduct}|A\cdot A|\geq|\pp^{\vv_1}B_{\vv_1}\cdot A'|\geq \sum_{\vv_2\in V'}|\pp^{\vv_1+\vv_2}B_{\vv_1}\cdot B_{\vv_2}|\geq \frac{1}{4^k}\sum_{\vv_2\in V'}M(\vv_1,\vv_2)\geq \frac{|A||B_{\vv_1}|}{4^{k+1}}.\end{equation}

This is a good estimate if $B_{\vv_1}$ is sufficiently large. If not, we must resort to growth from addition. We begin with the aforementioned lemma of Chang, which in the special case we need, requires no Fourier analysis.

\begin{Lemma}[Chang]\label{EasyChang} Let $A$ be a finite set of integers admitting a decomposition of the form 
\[ A=\bigcup_{\vv\in\NN_0^r} \pp^{\vv} B_\vv,\] where each $B_\vv$ is a finite set of integers coprime with $p_1,\ldots,p_r$. Then \[
E_+(A,A)^{1/2} \ll_r  \sum_{\vv} E_+(B_\vv,B_\vv)^{1/2}.\]
In particular,
\[E_+(A,A)\ll_r |A|^2\max_\vv |B_\vv|.\]
\end{Lemma}
\begin{proof} For convenience, replace $A$ with $A\cup -A$ so as to assume $A=-A$. Let $p=p_1$. Consider the equation 
\[a_1-a_2=a_3-a_4\,\]
with variables in $A$. Suppose, $a_1$ has the minimum $p$-adic valuation $v_p(a_1)$ among all $v_p(a_i)$. Then, by reducing $a_1=a_2+a_3-a_4$ modulo $p^{v_p(a_1)+1}$, we see a second term in the equation must have the same $p$-adic valuation, and at the cost of a constant factor (from rearrangement), this term is $a_2$.

It follows that 
\[E_+(A,A) \ll\sum_v E(p^v B_v, A)\ll  \sum_v E_+ (B_v, B_v)^{1/2} E_+(A,A)^{1/2},\]
 by Cauchy-Schwarz (applied to the additive energy) and cancelling $p^v$. Rearranging, and applying the resulting estimate for each prime from $\{p_1,\ldots,p_r\}$, we find
 \[E_+(A,A)^{1/2}\ll_r\sum_{\vv}E(B_\vv,B_\vv)^{1/2}.\]

 The final claim comes from applying the trivial estimate, \[E_+(B_\vv,B_\vv)\leq |B_\vv|^3\] 
 to each summand, whence
 \[E_+(A,A)^{1/2}\ll_r \max_{\vv}|B_\vv|^{1/2}\sum_{\vv}|B_\vv|=\max_{\vv}|B_\vv|^{1/2}|A|.\]
 \end{proof}

From Chang's Lemma applied to the set $A'$ from above, we find that \[|A+A|\geq |A'+A'|\gg_k \frac{|A'|^2}{|B_{\vv_1}|}\gg \frac{|A|^2}{|B_{\vv_1}|}.\] This, combined with (\ref{ThreeHalvesProduct}) we obtain the sum-product estimate with exponent $3/2$. Moreover, by tracking the dependence in the proof of Chang's Lemma, the implicit constant is singly exponential in $r$, which is at most $2k^2$.
\section{A refined structure theorem}

Our refined structure theorem will be similar in spirit to Corollary \ref{FibredStructure} but we shall seek to make the fibres $B_\vv$ identical. Such a claim is, in general, too good to be true, but can be made legitimate by \textit{covering} a subset $\tilde A\subseteq A$ by a structured set. Ultimately, we will estimate the additive energy of $\tilde A$, which is increasing on sets, and hence a covering will suffice. As concerns the multiplicative structure of $A$, it turns out in our case that we will want to regularize with respect to a dual decomposition of $A$ as shall be made explicit in the following lemma. 

\begin{Lemma}[Iteration Lemma]
    Suppose $A$ is a finite set of integers with the property $\omega(a)\leq k$ for $a\in A$. Let $p_1,\ldots,p_j$ be distinct primes, and suppose $A$ decomposes as \[A=\bigcup_{b\in B}b \cdot \Gamma_b,\]
    for some set $B$ of integers coprime to $p_1\cdots p_j$, and sets $\Gamma_b\subseteq \ang{p_1,\ldots, p_j}_+$.
    Then one of the following holds:
    \begin{enumerate}
        \item either $|B|\leq 2k$ or at least half of the pairs $(a,a')\in A\times A$ satisfy $\gcd(a,a')\in \ang{p_1,\ldots,p_j}$, or
        \item there is a prime $p_{j+1}$, distinct from $p_1,\ldots,p_j$, and a subset $\tilde A\subseteq A$ of size \[|\tilde A|\geq \frac{|A|}{2(k-j)\log_2 |A|},\] and having the form 
        \[\tilde A=\bigcup_{b\in \tilde B}b \cdot \tilde \Gamma_b\]
        for some set $\tilde B$ of integers coprime to $p_1\cdots p_{j+1}$, and sets $\tilde \Gamma_b\subseteq \ang{p_1,\ldots,p_{j+1}}_+$ with sizes satisfying $L\leq |\tilde \Gamma_b|\leq 2 L$ for some $ L\geq 1$.
    \end{enumerate}
\end{Lemma}
\begin{proof}
    Let $P_A$ denote the set of primes which divide some element of $A$, and let \[P=P_A\setminus\{p_1,\ldots,p_j\}.\] Consider the bipartite graph \[G=\{(a,p)\in A\times P:p|a\}\] and observe that  $\gcd(a,a')\not\in\ang{p_1,\ldots,p_j}$ if and only if $a$ and $a'$ have a common neighbour in $P$. Now, denoting $N_P(a)$ the neighbours of $a$ in $P$, we have $|N_P(a)|\leq k-j$, since each $a$ has at most $k$ prime factors, and $j$ of those are $p_1,\ldots,p_j$. Thus, writing $N_A(p)$ for the neighbours of $p$ in $A$, we find by double counting that
    \[(k-j)|A|\geq \sum_{a\in A}|N_P(a)|= \sum_{p\in P}|N_A(p)|\geq \frac{1}{\max_p |N_A(p)|}\sum_{p\in P}|N_A(p)|^2.\]
    If (1) fails and $|B|>2k$, then the rightmost sum above is at least $|A|^2/2$ and so there is some $p_{j+1}\in P$ with $|N_A(p_{j+1})|\geq |A|/2(k-j)$. Let $A'=N_A(p_{j+1})$ so that $A'$ further decomposes (by factoring out the appropriate powers of $p_{j+1}$) as
    \[A'=\bigcup_{b'\in B'}b'\cdot \tilde \Gamma_{b'}\]
    %%IS: b^{j+1} -> p_{j+1}
    where \[B'=\{p_{j+1}^{-v_{j+1}(b)}b:b\in B,\ p_{j+1}|b\},\] where $v_{j+1}(b)$ is the $p_{j+1}$-adic valuation of $b$.

    We now use dyadic pigeonholing:
    \[\frac{|A|}{2(k-j)}\leq \sum_{b'\in B'}|\tilde \Gamma_{b'}|\leq \sum_{l\leq \log_2 |A|}2^l|B_l'|\]
    where
    \[B_l'=\{b'\in B':2^l\leq |\tilde \Gamma_{b'}|<2^{l+1}\}.\]
    Thus, choosing a value of $l$ with maximal summand and setting $L=2^l$ and $\tilde B=B_l'$, we arrive at some subset $\tilde A\subseteq A'$ of the form \[\tilde A=\bigcup_{b\in \tilde B}b\cdot \tilde \Gamma_b\]
    of size at least $|\tilde A|\geq \frac{|A|}{2(k-j)\log_2 |A|}$, and with the property that each fibre $\tilde \Gamma_b$ has size in $[L,2 L]$.
\end{proof}

\begin{Theorem}\label{RegularFibredStructure}
    Suppose $A$ is a finite set of integers such what $\omega(a)\leq k$ for $a\in A$ and $|A\cdot A|\leq K|A|$. Then there is a set $P=\{p_1,\ldots,p_r\}$ of at most $r\leq k$ primes, and a set $\tilde A\subseteq A$ of size at least $|A|/(2^kk!(\log_2|A|)^k)$ and with the structural decompositions
    \[\tilde A= \bigcup_{\vv\in V}\pp^\vv \cdot B_\vv=\bigcup_{b\in B}b\cdot \Gamma_b,\]
    where each set $B_\vv$ a finite set of integers prime to $p_1\cdots p_r$, $B=\bigcup_{\vv\in V}B_\vv$ satisfies $|B|\leq 4^{k+2}K|A|/|\tilde A|$, and each set $\Gamma_b$ is a subset of $\ang{p_1,\dots,p_r}$ of size $L\leq |\Gamma_b|\leq 2L$ for an appropriate value of $L$.
\end{Theorem}
\begin{proof}
Iteratively apply the Iteration Lemma, beginning with $A_0=B_0=A$, $j=0$ and each $\Gamma_b=\{1\}$, continuing until conclusion (1) is satisfied. We obtain a sequence of distinct primes $p_1,\ldots,p_r$, a sequence of sets $A_0\supseteq \cdots \supseteq A_r$ such that $p_1\cdots p_i$ divides each element of $A_i$, and sets $B_1,\ldots,B_r$ of integers such that $B_i$ is coprime to $p_1\cdots p_i$. Here $A_{i+1}=\tilde{A_i}$, $B_{i+1}=\tilde B_i$ in the notation of conclusion (2) of the Lemma. Since $p_1\cdots p_i$ divides every element of $A_i$, and $B_i$ is coprime to $p_1\cdots p_i$, we have $\omega(b)\leq k-i$ for $b\in B_i$. Consequently, conclusion (1) must hold after $k$ iterations, so $r\leq k$. 

We set $B=B_r$ and $\tilde A=A_r$, so that we have \[\tilde A=\bigcup_{b\in B}b \cdot \Gamma_b\]
for some sets $\Gamma_b\subseteq \ang{p_1,\ldots, p_{j+1}}_+$ with sizes satisfying $L\leq |\Gamma_b|\leq 2L$ for some $L\geq 1$. We also have the estimate
\[|\tilde A|\geq \frac{|A|}{2^k k!(\log_2|A|)^k}.\] To get the other form of $\tilde A$ stated in the theorem, consider the decomposition \[\tilde A=\bigcup_{\vv\in V}\pp^{\vv}B_\vv\]
where $\pp=(p_1,\ldots,p_r)$, $V=\{(v_{p_1}(a),\ldots,v_{p_r}(a)):a\in \tilde A\}$ and \[B_\vv=\{b\in B:\pp^\vv\in \Gamma_b\}.\] 

We now estimate $|B|$. Now, the union $\tilde A=\bigcup_{b\in B}b\cdot \Gamma_b$ is disjoint and we have
\[\tilde A\cdot \tilde A=\bigcup_{q\in B\cdot B}q\cdot \bigcup_{bb'=q}\Gamma_b\cdot \Gamma_{b'},\] where the outer union is also disjoint. Indeed $B$ and $B\cdot B$, respectively, consist of coset representatives for $\tilde A$ and $\tilde A\cdot \tilde A$ with regard to the group $\ang{p_1,\ldots,p_r}$. Thus, from the trivial bound $|\Gamma_b\cdot \Gamma_{b'}|\geq L$, we see 
\[K|A|\geq |\tilde A\cdot \tilde A|\geq |B\cdot B|L\geq \frac{|B\cdot B||\tilde A|}{2|B|}.\]
Suppose $b,b'\in B$ are such that there is some pair $(a,a')\in (b\cdot \Gamma_b)\times (b'\cdot \Gamma_{b'})$ satisfying $\gcd(a,a')\in \ang{p_1,\ldots,p_r}$. Then it must be that $\gcd(b,b')=1$, and so 
\[|\{(b,b')\in B\times B:\gcd(b,b')=1\}|\geq \frac{1}{4L^2}|\{(a,a')\in \tilde A\times \tilde A:\gcd(a,a')\in\ang{p_1,\ldots,p_r}\}|.\]
Since the iteration can only have terminated because (1) was satisfied, it is either the case that $|B|\leq 2k$, which is stronger than promised, or that the right hand side above is at least $|\tilde A|^2/8L^2$. In the latter case Lemma \ref{ProductBound} tells us that 
%%IS: 4^{k+1} -> 8 4^k
\[|B\cdot B|\geq \frac{|\tilde A|^2}{8L^24^k}\geq \frac{|B|^2}{8\cdot 4^{k}},\]
and hence
\[K|A|\geq \frac{|B||\tilde A|}{4^{k+2}}.\]
\end{proof}

\section{Fourier Analysis}

Chang's estimate, Lemma \ref{EasyChang}, could be described as an estimate for the $\Lambda(q)$ constant for subsets of a multiplicative group generated by a finite set of primes (often called $S$-units). However, the estimate is a bit crude in some cases, the result of an application of H\"older's inequality which is at times inefficient. The next ingredient in our proof is a square-function estimate for a Littlewood-Paley decomposition along a sequence of multiples. The ultimate goal of this section is to prove the following theorem.
\begin{Theorem}\label{SquareFunction}
    Let $\pp=(p_1,\ldots,p_r)$ be an $r$-tuple of distinct primes and let $f(t)=\sum_{a\in A}\hat f(a)e(at)$ be a trigonometric polynomial whose Fourier coefficients are supported in a set $A$ of the form
    \[A=\bigcup_{\vv\in V}\pp^\vv B_{\vv},\]
    where each set $B_\vv$ is a set of integers coprime to $p_1\cdots p_r$.
    Define \[f_\vv(t)=\sum_{b\in B_\vv}\hat f(\pp^\vv b)e(\pp^\vv bt).\]
Then for any $q$ with $1<q<\infty$, there is a constant $C_q>0$ such that
\[\|f\|_{L^{q}}\leq C_q^r\left\|\lr{\sum_{\vv}\left|f_\vv\right|^2}^{1/2}\right\|_{L^{q}}.\]
\end{Theorem}

At least when $r=1$ this 
%%IS
%theorem 
result 
is a fairly straightforward consequence of Burkholder's martingale Littlewood-Paley theorem. Each of \cite{EG}, \cite{Pisier}, and \cite{Stein} have readable expositions. In fact, the conclusion above was previously observed as a consequence of Burkholder's theorem by Gundy and Varapoulos in \cite{GV}. Experts familiar with such square-function estimates should feel free to skip the rest of this section, where we present some of the salient points of the proof. This exposition is presented partly for the sake of completeness, although we do not give a proof of Burkholder's theorem, and in order to have the necessary facts combined in a single source (they can be also located in various parts of \cite{EG}). Beyond specializing these facts to the application at hand, no originality is claimed. 

Applying Theorem \ref{SquareFunction} with $q=4$ and $\hat f=\one_A$, we have, we have the following refinement to Chang's energy estimate.
\begin{Corollary}\label{EnergyDecomposition}
    Let $A\subseteq \ZZ$ be a finite set of the form
    \[A=\bigcup_{\vv\in V}\pp^\vv B_{\vv}.\]
    Then 
    %%IS
    there is a constant $C>0$ 
    \[E(A,A)\leq C^r \sum_{\vv,\vv'\in V}E(\pp^\vv B_{\vv},\pp^{\vv'} B_{\vv'}).\]
\end{Corollary}

Corollary \ref{EnergyDecomposition}, when coupled with the Cauchy-Schwarz inequality, recovers Chang's original estimate. However, we will see in the next section, that for many pairs $(\vv_1,\vv_2)$, we have a substantial improvement on the trivial energy estimate.

As mentioned above, in order to prove Theorem \ref{SquareFunction}, we will make use of Burkholder's inequalities for martingale transforms. Specifically, we will use the following result which bounds the norm of multipliers which are constant on $p$-adic scales. In what follows, we write \[f_\eps(t)=\sum_{n\in \ZZ}\eps(n)\hat f(n)e(nt).\] 
\begin{Theorem}\label{Burkholder}
    Let $p$ be a prime and suppose $q$ is such that $1<q<\infty$. Let $\eps:\ZZ\to \{-1,1\}$ be a function such that $\eps(n)$ depends only on $v_p(n)$. Then there is an absolute constant $C_q$, depending only on $q$ such that we have
    \[\left\|f\right\|_{L^{q}}\leq C_q\left\|f_\eps\right\|_{L^{q}}.\]
\end{Theorem}

One should think of choosing $\eps$ to be random, subject to the constraint that it be constant on $p$-adic scales. Then, the multiplier theorem above is seen to be equivalent to the square-function estimate quoted in Theorem \ref{SquareFunction} (in the case $r=1$) by way of Khintchine's inequality. First some notation: for a partition $\cP$ of $\ZZ$, write
\[(S_\cP f)(t)=\lr{\sum_{P\in\cP}\left|\sum_{n\in P}\hat f(n)e(nt)\right|^2}^{1/2}.\]
\begin{Lemma}\label{Khintchine}
    Let $\cP$ be a partition of $\ZZ$. Then the following are equivalent:
    \begin{enumerate}
        \item for any $q> 1$ there are constants $c_q$ and $C_q$ such that for any trigonometric polynomial $f$,
        \[c_q\|f\|_{L^{q}}\leq \left\|S_{\cP}f\right\|_{L^{q}}\leq C_q\|f\|_{L^{q}},\]
        \item for any $q> 1$ there is a constant $C_q$ such that for any trigonometric polynomial $f$ and any function $\eps:\ZZ\to\{-1,1\}$ which is constant on the parts of $\cP$, we have
        \[\left\|f\right\|_{L^{q}}\leq C_q\left\|f_\eps\right\|_{L^{q}}.\]
    \end{enumerate}
\end{Lemma}
\begin{proof}
    Let \[f_\eps(t)=\sum_n\eps(n)\hat f(n)e(nt).\] Then assuming clause (1) of the lemma, \[
     \|f\|_{L^{q}}\leq \frac{1}{c_q}\left\|S_\cP f\right\|_{L^{q}}=\frac{1}{c_q}\left\|S_\cP f_\eps\right\|_{L^{q}}\leq \frac{C_q}{c_q}\|f_\eps\|_{L^{q}}. \]
    Conversely, if we assume clause (2) of the lemma, then we can write $f=(f_{\eps})_\eps$ so the reverse inequality \[\|f_\eps\|_{L^q}^q\leq C_q^q \|f\|_{L^q}^q\] holds, and taking expectation over all choices of $\eps$, \[C_q^q\|f\|_{L^q}^q\geq \EE_\eps \|f_\eps\|_{L^{q}}^{q}= \int_0^1 \EE_\eps \left|\sum_{P\in \cP}\eps(P)\sum_{n\in P}\hat f(n)e(nt)\right|^q dt.\] We get from Khintchine's inequality (see, for instance, Lemma 5.5 of \cite{MS}) that
    \[\EE_\eps \left|\sum_{P\in \cP}\eps(P)\sum_{n\in P}\hat f(n)e(nt)\right|^q \geq c_q'\lr{\sum_{P\in\cP}\left|\sum_{n\in P}\hat f(n)e(nt)\right|^2}^{q/2}=c_q'(S_\cP f)^q,\]
    which proves the second inequality from (1) upon integration over $t\in[0,1]$.

    To prove the first inequality in clause (1), we appeal to duality. Let $q'$ be the dual exponent to $q$ and 
    %IS
    %suppose 
    take a trigonometric polynomial $g$ such that 
    $\|g\|_{L^{q'}}=1$. By orthogonality and the triangle inequality,
    \[|\langle f,g\rangle| \leq \int_0^1 \sum_{P\in\cP}\left|\sum_{n\in P}\hat f(n)e(nt)\right|\left|\sum_{n\in P}\hat g(n)e(nt)\right|dt.\]
    By the Cauchy-Schwarz inequality, the right hand side is at most
    \[\left\langle S_\cP f,S_\cP g\right\rangle\leq \|S_\cP f\|_{L^q}\|S_\cP g\|_{L^{q'}}\leq C_{q'}\left\|S_\cP f\right\|_{L^q}\]
    where in the last estimate we have applied H\"older's inequality and the second inequality from clause (1) to $g$.
\end{proof}

Here we remark that, using the last part of the above proof, it will generally suffice to prove the second inequality from (1), whence the first can be derived from duality.

The reason for introducing the multiplier formulation is that it is well-suited to iteration, allowing us to prove the following.

\begin{Lemma}\label{Iteration}
    Let $q\geq 1$ and suppose $\cP_1$ and $\cP_2$ are two partitions of $\ZZ$ such that for any trigonometric polynomial $f$,
    \[\left\|S_{\cP_j}f\right\|_{L^{q}}\leq C_{q}(\cP_j)\left\|f\right\|_{L^{q}} \,, \quad \quad j=1,2\,.\] Then if $\cP=\{P_1\cap P_2:P_1\in\cP_1,P_2\in \cP_2\}$, there is a constant $C_q(\cP)$ such that \[\left\|S_{\cP}f\right\|_{L^{q}}\leq C_{q}(\cP)\left\|f\right\|_{L^{q}}\]
\end{Lemma}
\begin{proof}
    First assume $q\geq 2$. Let $\eps:\ZZ\to \{-1,1\}$ be a function which constant on the parts of $\cP_2$, and suppose $f$ is a trigonometric polynomial. By hypothesis and Lemma \ref{Khintchine}, there are is a positive constant $C$ such that
    \[\|f_\eps\|_{L^q}\leq C\|f\|_{L^q},\]
    whence \[\|S_{\cP_1}f_\eps\|_{L^q}\leq C_{q}(\cP_1)\|f_\eps\|_{L^q}\leq C_{q}(\cP_1)C\|f\|_{L^q}.\]
    Now
    \[(S_{\cP_1}f_\eps(t))^2=\sum_{P_1\in \cP_1}\sum_{n,m\in P_1}\eps(n)\eps(m)\hat f(n)\bar{\hat f(m)}e((n-m)t),\]
    and taking expectation over $\eps$ yields
    \[\EE_\eps((S_{\cP_1}f_\eps(t))^2)=\sum_{P_1\in \cP_1}\sum_{P_2,P_2'\in \cP_2}\sum_{\substack{n,m\in P_1\\ n\in P_2,m\in P_2'}}\EE_\eps(\eps(n)\eps(m))\hat f(n)\bar{\hat f(m)}e((n-m)t).\]
    The expectation vanishes unless $P_2=P_2'$, in which case it is 1, and hence 
    \[\EE_\eps((S_{\cP_1}f_\eps(t))^2)=\sum_{P_1\in\cP_1}\sum_{P_2\in \cP_2}\sum_{n,m\in P_1\cap P_2}\hat f(n)\bar{\hat f(m)}e((n-m)t)=(S_{\cP}f(t))^2.\]
    Raising to the power $q/2$, we find
    \[(S_{\cP} f(t))^q=(\EE_\eps((S_{\cP_1}f_\eps(t))^2)^{q/2}\leq \EE_{\eps}((S_{\cP_1}f_\eps(t))^q)\]
    by Jensen's inequality, and integrating over $t$ shows
    \[\|S_{\cP}f\|_{L^q}^q\leq (C_{\cP_1}C)^q\|f\|_{L^q}^q,\]
    as required.

    To get the claim for $1\leq q<2$, we use duality and the random multiplier formulation. Indeed, let $q'\geq 2$ be the exponent conjugate to $q$, let $\eps:\ZZ\to\{-1,1\}$ be a function which is constant on the parts of $\cP$, and let $g$ be a trigonometric polynomial with $\|g\|_{L^{q'}}=1$. Then by Parseval and H\"older's inequality,
    \[|\ang{f_\eps,g}|=|\ang{f,g_\eps}|\leq \|f\|_{L^q}\|g_\eps\|_{L^{q'}}\leq C_{q'}(\cP)\|f\|_{L^q}\]
    which shows $\|f_\eps\|_{L^q}\leq C_{q}\|f\|_{L^q}$ for any $\eps:\ZZ\to\{-1,1\}$ which is constant on the parts of $\cP$, and hence the boundedness of $S_\cP$ follows from Lemma \ref{Khintchine}.    
\end{proof}

\begin{proof}[Proof of Theorem \ref{SquareFunction}]
    To each prime $p_i$ with $1\leq i\leq r$ we associate the partition $\cP_i$ of $\ZZ$ into $p_i$-adic scales. From Theorem \ref{Burkholder} and Lemma \ref{Khintchine}, we see that for each prime $p_i$ with $1\leq i\leq r$, we find a constant $C_q$ such that $\|S_{\cP_i}f\|_{L^q}\leq C_q\|f\|_{L^q}$. From duality, it suffices to show that the common refinement of the partitions $\cP_i$ yields a bounded square function. This in turn follows from Lemma \ref{Iteration} applied $r-1$ times.
\end{proof}

\section{Energy estimates with Dilates}

In this section we prove estimates for the number of differences which lie in fixed (coset of a) multiplicative group of bounded rank. When the rank is one, this can be achieved in an elementary fashion as described by the following theorem. This theorem will not be needed, unless $k=2$ (although this special case, as has been discussed at the outset is already quite nontrivial) but we include it as it may be of independent 
interest 
to prove our results without an appeal to much more  sophisticated results.

\begin{Theorem}\label{RankOne}
Let $B$ be a finite set of positive integers with $|B|\geq 2$, let $p$ be a prime, and let $n$ be a non-zero integer. Define
\[X_p(B,n)=\{(b_1,b_2)\in B\times B: b_1-b_2=np^v\text{ for some }v\in \ZZ_{\geq 0}\}.\]
Then $|X_p(B,n)|\leq 1+4|B|\log_2|B|$.
\end{Theorem}
\begin{proof}
Let $p$ be a fixed prime. It will be convenient to normalize $B$ as follows. First, if $p$ divides $n$ then we write $n=p^{v}n'$ with $\gcd(n',p)=1$. Then $X_p(B,n)\subseteq X_p(B,n')$ and so there is no loss of generality in assuming $\gcd(n,p)=1$. Next, if $B$ lies in a single congruence class modulo $p^{r_0}$ for some $r_0>0$, then we may replace $B$ with $B-\min B$ (or any other element of $B$), without affecting $|X_p(B)|$. The result would be that $B-\min B$ consists of multiples of $p^{r_0}$, and since $p$ does not divide $n$, we can then bound $|X_p(B,n)|$ by $|X_p(p^{-r_0}B,n)|$. So we may further assuming the elements of $B$ are not all congruent modulo $p$.

We proceed by induction on $|B|$. When $|B|=2$, suppose $B=\{b,b'\}$ is a set and $p$ is a prime. Then $b-b'=p^rn$ for at most one value of $r$, so $|X_p(B)|\leq 1$ and this establishes the base case. For larger $B$, we condition on the value of $b\mod p$. To do that, we write \[B_u=\{b\in B:b\equiv u\mod p\}\] and let $\mu(u)=\frac{|B_u|}{|B|}$ be the accompanying probability measure. Given $b\equiv u\mod p$, we either have $b-b'=n$ in which case $b'\equiv u-n\mod p$, and such solutions contribute at most
\begin{align*}
    \sum_{u\mod p}\min\{|B_u|,|B_{u-n}|\}&= |B|\sum_{u\mod p}\min\{\mu(u),\mu(n-u)\}\\
    &\leq |B|\sum_{u\mod p}\min\{\mu(u),1-\mu(u)\}\\
    &=|B|\min\{1,2-2\mu(u_{\max})\}
\end{align*} solutions, where $u_{\max}$ is the residue class for which $\mu(u)$ is largest. Otherwise $b-b'=p^rn$ for some $r>0$ in which case $b$ and $b'$ agree modulo $p$. Thus
\[|X_p(B,n)|\leq \sum_{u\mod p}|X_p(B_u,n)|+|B|\min\{1,2-2\mu(u_{\max})\}.\]
Since $B_u\neq B$ by our normalization, we apply induction and find \[|X_p(B_u,n)|\leq |B_u|(1+4\log_2|B_u|)=\mu(u)|B|(1+4\log_2|B|)-4\mu(u)|B|\log(1/\mu(u)).\] Putting this all together,
\[|X_p(B,n)|\leq |B|(1+4\log_2|B|)-4|B|(H(\mu)-\frac{1}{4}\min\{1,2-2\mu(u_{\max})\}).\]
Here \[H(\mu)=\sum_{u\mod p}\mu(u)\log_2\frac{1}{\mu(u)}\]
is the entropy of the measure $\mu$. We claim \[H(\mu)-\frac{1}{4}\min\{1,2-2\mu(u_{\max})\})\geq0.\]
Indeed, if $\mu(u_{\max})\leq 1/2^{1/4}$ then $H(\mu)\geq 1/4$, otherwise from the inequality \[\frac{1-x}{2}\leq \log_2 \frac{1}{x},\] we have
\[\frac{1-\mu(u_{\max})}{2}\leq \log_2\frac{1}{\mu(u_{\max})}= \sum_{u\mod p}\mu(u)\log_2\lr{\frac{1}{\mu(u_{\max})}}\leq \sum_{u\mod p}\mu(u)\log_2\lr{\frac{1}{\mu(u)}}.\]
\end{proof}

It may be that the above argument extends to the case of higher rank  $r>1$. However, we can just overwhelm the problem with some heavy machinery from the theory of $S$-unit equations. The following argument uses a quantitative estimate concerning linear equations in a multiplicative group, taken from \cite{AV}, improving the work of  \cite{ESS}. 
\begin{Theorem}[$S$-unit bound] Let $S=\{p_1,\ldots,p_r\}$ be a set of rational primes and let $\Gamma=\ang{S}$ be the multiplicative group they generate.
For fixed $a_1,\ldots,a_l\in \mathbb C^\times$, 
\[
\left|\left\{\gamma_1,\ldots,\gamma_l\in \Gamma: \sideset{}{^*}\sum_{1\leq i\leq l} a_i\gamma_i=1\right\}\right| \leq (8l)^{4l^2+lr+1},
\]
where the notation $\sum^*$ indicates non-degeneracy in the sense that  $\sum_{i\in I}a_i\gamma_i\neq 0$ for non-empty proper subsets $I\subset\{1,\ldots,l\}$.
\label{th:subspace}
\end{Theorem}

We need this estimate for the following application taken from \cite{RNZ}, see Lemma 2.1 therein. We include the proof so as to be quantitatively explicit. We quote the result for rational numbers, although it applies much more broadly.
 
\begin{Lemma}[Lemma 2.1 of \cite{RNZ}]\label{RNZ}
    Suppose $0<\eps<\frac{1}{6}$. For any sufficiently large set $B$ of rational numbers and a multiplicative group $\Gamma\subseteq\CC^\times$ generated by $r$ rational primes such that $r\leq (\log |B|)^{1-6\eps}$, one has the estimate
    \[|\{(b_1,b_2)\in B\times B:\,b_1-b_2\in \Gamma\}|\leq |B|\exp((\log|B|)^{1-\eps}).\]
    \label{l:subspace_app}
\end{Lemma}
\begin{proof}
For ease of notation, let $|B|=n$, and if necessary, augment $\Gamma$ by adjoining $-1$ to it. Consider the undirected graph $G$ on the vertex set $B$, whose edges are those $\{b_1,b_2\}$ satisfying $b_1-b_2\in \Gamma$. Let the number of edges be denoted by $nf(n)$, for some function $f(n)$. One can assume that $f(n)$ is increasing and larger than $(\log n)^r$, or else there is nothing to prove.

Let $d=f(n)/2$. We first prune $G$ by iteratively removing vertices with degree less than $d$, updating the degrees of the vertices (but not the threshold $d$) after each stage to reflect any removal. This process must terminate after at most $n$ steps as there are at most $n$ vertices that can be removed, and when it does terminate, we can have removed no more than $n\cdot f(n)/2$ edges. If necessary, we redefine $G$ to be the pruned graph, in which each vertex has degree at least $f(n)/2$

Fix a vertex $b_0$, and consider a non-degenerate path of length $l$ in $G$, starting from $b_0$. The path $b_0,b_1,\ldots,b_l$ corresponds to the telescopic sum
\[
(b_1-b_0)+(b_2-b_1)+\ldots+(b_l-b_{l-1}) = \gamma_1 + \gamma_2 +\ldots +\gamma_l,\]
and we call the path non-degenerate if no subsum of the right-hand side vanishes. Given a non-degenerate path of length $l$, one can append to it at least $f(n)/2 - (2^l-1)$ edges and get a non-degenerate path of length $l+1$. Indeed, there are only $2^l-1$ edges that could lead to a degeneracy. Hence, we find by way of induction that the number of non-degenerate paths of length $l$ is at least $f(n)^l/4^l,$ provided that $f(n)\geq 2^{l+2}$.

So, there are at least $f(n)^{l-1}/(4^{l-1} n)$ non-degenerate paths between $b_0$ and some other element $b\in B$, and so $f(n)^{l}/(4^{l} n)$ paths from $b_0$ to some $b_1\in B$, by appending an edge $b-b_1$ to said path.
On the other hand, Theorem \ref{th:subspace} provides the upper bound for this number of paths, once one chooses $a_i = 1/(b_1-b_0)$ for $1\leq i\leq l$. Taking logarithms and assuming that $r\geq 2$ and, say $l\geq 100$, simplifies the upper bound to
\[l\log f(n) \leq l^{\frac{11}2} r + \log n\]
Upon choosing $l \approx \left(\frac{\log n}{r}\right)^{\frac{2}{11}}$ to balance the terms of the right-hand side, we conclude
\[\log f(n) \ll {r}^{\frac{2}{11}} (\log n)^{\frac{9}{11}},\]
which completes the proof in view of the bound on $r$ assumed in the statement of the lemma.
\end{proof}
Observe that in Lemma \ref{l:subspace_app}, the condition $b_1-b_2\in \Gamma$ can be replaced by a coset membership $b_1-b_2\in u\cdot \Gamma$ by dilating $B$.

\section{The proof of Theorem \ref{MainThm}}

Let $A$ be a finite set of integers such that $\omega(a)\leq k$ for each $a\in A$. By passing to a subset and dilating by $-1$ if necessary, we may assume that $A\subset \NN$, at the cost of a constant factor. Set $K=|A\cdot A|/|A|$ and apply Theorem \ref{RegularFibredStructure} to obtain a set and a set $\tilde A\subseteq A$ of size \begin{equation}\label{eqn:ABound}
    |\tilde A|\geq \frac{|A|}{2^kk!(\log_2|A|)^k}\geq \frac{|A|}{(\log |A|)^{3k}},
\end{equation} by the hypothesized bounds on $k$, and having the structure
    \[\tilde A= \bigcup_{\vv\in V}\pp^\vv \cdot B_\vv=\bigcup_{b\in B}b\cdot \Gamma_b,\]
    where each $B_\vv$ is a finite set of integers prime to $p_1\cdots p_r$ such that their union $B=\bigcup_{\vv\in V}B_\vv$ satisfies \begin{equation}\label{eqn:BBound}\left|B\right|\leq \frac{4^{k+2}K|A|}{|\tilde A|},\end{equation}
    while each $\Gamma_b$ is a subset of $\Gamma=\ang{p_1,\ldots,p_r}$ with size satisfying $L\leq |\Gamma_b|\leq 2L$ for some appropriate $L$.

We now estimate the additive energy of $\tilde A$, and in doing so, we may assume that $\tilde A=-\tilde A$, at the cost of a constant. By Corollary \ref{EnergyDecomposition}, we have that 
\[E(\tilde A,\tilde A)\leq C^k\sum_{\vv_1,\vv_2\in V}E(\pp^{\vv_1} B_{\vv_1},\pp^{\vv_2}B_{\vv_2}).\]
The sum above counts solutions in $\tilde A$ to the equation
\[a_1-a_2=a_3-a_4,\]
where $a_1=b_1\pp^{\vv_1}$, $a_2=b_2\pp^{\vv_1}$ for some $b_1,b_2\in B_{\vv_1}$ and $a_3=\pp^{\vv_2}b_3$, $a_4=\pp^{\vv_2}b_4$ for some $b_3,b_4\in B_{\vv_2}$. In other words, we have reduced to the case were the exponents appearing on the left and right of the energy equation are both repeated. Let us now write $\gamma=\pp^{\vv_1}$ and $\gamma'=\pp^{\vv_2}$ so we are left counting solutions to \begin{equation}\label{eqn:reduced}
    b_1-b_2=\gamma^{-1}\gamma'(b_3-b_4),
\end{equation}
where now $b_1,\ldots,b_4\in B$, $\gamma\in \Gamma_{b_1}\cap \Gamma_{b_2}$ and $\gamma'\in \Gamma_{b_3}\cap \Gamma_{b_4}$. The only solutions where $b_3=b_4$ correspond to trivial solutions to the energy equation in $\tilde A$, of which there are at most $|\tilde A|^2$. For the remaining solutions, fix $b_3$, $b_4$ and $\gamma'$, and then observe that the number of solutions to (\ref{eqn:reduced}) is at most $O(|B|\exp((\log|B|)^{1-\eps}))$ by Lemma \ref{RNZ}. There are at most $|B|$ choices for $b_4$ and at most \[\sum_{b\in B}|\Gamma_b|=|\tilde A|\] choices for the pair $(b_3,\gamma')$. Putting all of this together, we see
\[E(\tilde A,\tilde A)\ll C^k(|\tilde A|^2+|\tilde A||B|^2\exp((\log|B|)^{1-\eps}))\ll C^k\lr{|\tilde A|^2+K^2|\tilde A|\exp(2(\log|A|)^{1-\eps})}\]
upon inserting the appropriate bounds for $|B|$ and $|\tilde A|$ coming from (\ref{eqn:BBound}) and (\ref{eqn:ABound}). If the quantity $C^k|\tilde A|^2$ dominates then we have proved more than enough. If not, then from the bound (\ref{eqn:ABound}), it would suffice to prove
\[K|A|+\frac{|A|^3}{K^2}\gg |A|^{\frac{5}{3}}\] which is now obvious since the left hand size is minimized when $K\approx |A|^{2/3}$.


\begin{thebibliography}{99}

\bibitem[AV]{AV} F. Amoroso and E. Viada, {\em Small points on subvarieties of a torus.} Duke Math. J. 150 (2009), no. 3, 407-442.

\bibitem[BW]{BW}
A. Balog and T. D. Wooley, \textit{A low-energy decomposition theorem.} Q. J. Math. 68 (2017), no. 1, 207-226.

\bibitem[BC]{BC}
J. Bourgain and M.-C. Chang, \textit{On the size of k-fold sum and product sets of integers.} J. Amer. Math. Soc. 17 (2004), no. 2, 473–497.

\bibitem[Bu]{Bur}
D. L. Burkholder, \textit{Martingale transforms.} Ann. Math. Statist. 37 (1966), 1494-1504.

\bibitem[Ch]{Chang} M.-C. Chang, \textit{The Erd\H{o}s-Szemer\'edi problem on sum set and product set.} Ann. of Math. (2) 157 (2003), no. 3, 939-957. 

\bibitem[EG]{EG}
R. E. Edwards and G. I. Gaudry, \textit{Littlewood-Paley and multiplier theory.} Ergebnisse der Mathematik und ihrer Grenzgebiete, Band 90. Springer-Verlag, Berlin-New York, 1977. ix+212 pp.

\bibitem[E]{E}
G. Elekes, \textit{On the number of sums and products}. Acta Arith. 81 (1997), no. 4, 365-367.

\bibitem[ES]{ES} P. Erd\H{o}s and E. Szemer\'edi, \textit{On sums and products of integers.} Studies in pure mathematics, 213-218, Birkh\"auser, Basel, 1983.

\bibitem[ESS]{ESS} J. H. Evertse,  H. P. Schlickewei and W. M. Schmidt, 
{\em Linear equations in variables which lie in a multiplicative group.}
Ann. of Math. (2) 155 (2002), no. 3, 807-836.

\bibitem[GT]{GT}
B. Green and T. Tao, \textit{The primes contain arbitrarily long arithmetic progressions.} Ann. of Math. (2) 167 (2008), no. 2, 481–547.

\bibitem[GV]{GV} R. F. Gundy and N. Th. Varopoulos, \textit{A martingale that occurs in harmonic analysis.} Ark. Mat. 14 (1976), no. 2, 179-187.

\bibitem[HRNR]{HRNR}
B. Hanson, O. Roche-Newton and M. Rudnev, \textit{Higher convexity and iterated sum sets.} Combinatorica 42 (2022), no. 1, 71-85.

\bibitem[HRNZ1]{HRNZ1}
B. Hanson, O. Roche-Newton and D. Zhelezov, \textit{On iterated product sets with shifts.} Mathematika 65 (2019), no. 4, 831–850.

\bibitem[HRNZ2]{HRNZ2}
B. Hanson, O. Roche-Newton and D. Zhelezov, \textit{On iterated product sets with shifts, II.} Algebra Number Theory 14 (2020), no. 8, 2239-2260.

\bibitem[KR]{KR} 
S. V. Konyagin and M. Rudnev, {\em On new sum-product-type estimates.} SIAM J. Discrete Math. 27 (2013), no. 2, 973-990.

\bibitem[MSt]{MSt}
A. Mohammadi and S. Stevens, 
\textit{Attaining the exponent $5/4$ for the sum-product problem in finite fields.} International Mathematics Research Notices 2023, no. 4 (2023): 3516-3532.

\bibitem[MV]{MV}
H. L. Montgomery and R. C. Vaughan, \textit{Multiplicative number theory. I. Classical theory.} Cambridge Studies in Advanced Mathematics, 97. Cambridge University Press, Cambridge, 2007. xviii+552 pp.

\bibitem[MRSS]{BRS} B. Murphy, M. Rudnev, I. Shkredov and Yu. Shteinikov, {\em On the few products, many sums problem.} J. Th\'eor. Nombres Bordeaux 31 (2019), no. 3, 573-602. 

\bibitem[MS]{MS}
C. Muscalu and W. Schlag, \textit{Classical and multilinear harmonic analysis. Vol. I.} Cambridge Studies in Advanced Mathematics, 137. Cambridge University Press, Cambridge, 2013. 

\bibitem[N]{N}
 M. B. Nathanson, \textit{On sums and products of integers.} Proc. Amer. Math. Soc. 125 (1997), no. 1, 9-16.

\bibitem[PZ]{PZ}
D. P\'alv\H{o}lgyi and D. Zhelezov, \textit{Query complexity and the polynomial Freiman-Ruzsa conjecture.} Adv. Math. 392 (2021), Paper No. 108043, 18 pp.

\bibitem[P]{Pisier} G. Pisier, \textit{Martingales in Banach spaces.} Cambridge Studies in Advanced Mathematics, 155. Cambridge University Press, Cambridge, 2016. xxviii+561 pp. 

\bibitem[RNZ]{RNZ}O. Roche-Newton and D. Zhelezov, \textit{A bound on the multiplicative energy of a sum set and extremal sum-product problems.} Mosc. J. Comb. Number Theory 5 (2015), no. 1-2, 52-69.

\bibitem[RSh]{RSh} M. Rudnev and I. D. Shkredov, \textit{On the growth rate in $\mathrm{SL}_2(\FF_p)$, the affine group and sum-product type implications. }Mathematika 68 (2022), no. 3, 738-783.

\bibitem[RSt]{RS}M. Rudnev and S. Stevens, \textit{An update on the sum-product problem.} Math. Proc. Cambridge Philos. Soc. 173 (2022), no. 2, 411-430.

\bibitem[So]{S}
J. Solymosi, \textit{Bounding multiplicative energy by the sumset.} Adv. Math. 222 (2009), no. 2, 402-408.

\bibitem[St]{Stein}
E. M. Stein, \textit{Topics in harmonic analysis related to the Littlewood-Paley theory.} Annals of Mathematics Studies, No. 63 Princeton University Press, Princeton, N.J.; University of Tokyo Press, Tokyo 1970 viii+146 pp.

\bibitem[TV]{TaoVu} T. Tao and V. Vu, \textit{Additive combinatorics.} Cambridge Studies in Advanced Mathematics, 105. Cambridge University Press, Cambridge, 2006. xviii+512 pp.

\end{thebibliography}
\end{document}